\documentclass[11pt, oneside]{article}
\usepackage{amsthm}
\usepackage{amsmath}
\usepackage{amssymb,authblk}
\usepackage{t1enc}
\usepackage[utf8]{inputenc}
\usepackage{float}
\usepackage{tikz}
\usepackage{cite}
\usepackage[nobysame]{amsrefs}
\usepackage[unicode,colorlinks=true,citecolor=green!40!black,linkcolor=red!20!black,urlcolor=blue!40!black,filecolor=cyan!30!black]{hyperref}
\usetikzlibrary{calc}
\usetikzlibrary{decorations}
\usetikzlibrary{decorations.pathreplacing}
\usepackage{fullpage}
\usepackage{enumerate}
\usepackage{enumitem}
\usepackage[textsize=footnotesize, textwidth=2cm, color=yellow]{todonotes}

\usepackage[capitalise,nameinlink,sort]{cleveref}
\crefname{condition}{condition}{conditions}
\Crefname{condition}{Condition}{Conditions}

\makeatletter
\DeclareRobustCommand{\crefnosort}[1]{%
  \begingroup\@cref@sortfalse\cref{#1}\endgroup
}
\makeatother

\newtheorem{theorem}{Theorem}[section]
\newtheorem{lemma}[theorem]{Lemma}
\newtheorem{observation}[theorem]{Observation}
\newtheorem{proposition}[theorem]{Proposition}
\newtheorem{corollary}[theorem]{Corollary}
\newtheorem{conjecture}[theorem]{Conjecture}
\newtheorem{claim}{Claim}

\theoremstyle{definition}
\newtheorem{definition}[theorem]{Definition}

\makeatletter
 \newcommand{\linkdest}[1]{\Hy@raisedlink{\hypertarget{#1}{}}}
\makeatother

\title{On minimally tough chordal graphs}

\author[1]{Cl\'ement Dallard}
 \author[2,3]{Blas Fern\'andez}
\author[4,6]{Gyula Y.~Katona}
\author[2,3]{Martin Milani\v{c}}
\author[4,5]{Kitti Varga}

\affil[1]{Department of Informatics, University of Fribourg, Switzerland}
\affil[2]{FAMNIT, University of Primorska, Koper, Slovenia}
\affil[3]{IAM, University of Primorska, Koper, Slovenia}
\affil[4]{Department of Computer Science and Information Theory, Budapest University of Technology and Economics, Budapest, Hungary}
\affil[5]{HUN-REN--ELTE Egerv\'{a}ry Research Group, Budapest, Hungary}
\affil[6]{HUN-REN–ELTE Numerical Analysis and Large Networks Research Group, Budapest, Hungary}

\date{}

\begin{document}

\maketitle

\begin{abstract}
Katona and Varga showed that for any rational number $t \in (1/2,1]$, no chordal graph is minimally $t$-tough,
while Katona and Khan characterized all minimally $t$-tough, chordal graphs with $t \le 1/2$.
We conjecture that no chordal graph is minimally $t$-tough for any $t>1$ and prove several results supporting the conjecture.
In particular, we show that for any $t>1/2$, no strongly chordal graph is minimally $t$-tough
and no chordal graph with a universal vertex is minimally $t$-tough.

\medskip

\end{abstract}

\section{Introduction}

In 1973, Chv\'{a}tal defined a (finite, simple, and undirected) graph $G$ to be \hbox{$t$-tough}, where $t$ is a real number, if the cardinality of each vertex set $S$ whose removal disconnects the graph is at least $t$ times the number of components of $G-S$.
The toughness of a graph $G$ is the largest real number $t$ for which $G$ is $t$-tough, whereby the toughness of complete graphs is defined as infinity \cite{article:toughness_introduction}. As informally described by Chv\'{a}tal himself, toughness ``measures in a simple way how tightly various pieces of a graph hold together''.
Toughness was introduced to generalize the notion of Hamiltonicity since Hamiltonian graphs are $1$-tough (but not every 1-tough graph is Hamiltonian). Chv\'{a}tal conjectured that there exists a real number $t_0$ such that every $t_0$-tough graph is Hamiltonian~\cite{article:toughness_introduction}.
He also proposed a stronger conjecture stating that every graph with toughness greater than $3/2$ is Hamiltonian, but this was disproved by Thomassen~\cite{article:3/2-disproof}. Thereafter it was conjectured (based on~\cite{article:2factor}) that every 2-tough graph is Hamiltonian -- but this was also disproved, this time by Bauer et al.~\cite{article:9/4}. The general conjecture, nevertheless, is still open.

A concept closely related to Chv\'{a}tal's conjecture is that of minimally $t$-tough graphs, which Broersma et al.~defined as graphs whose toughness is $t$ but the deletion of any edge decreases their toughness~\cite{article:min_tough}. 
This notion has been studied in a number of subsequent works (see, e.g.,~\cites{article:dp, article:min1toughgraphs, article:spec_graph_classes_journal, article:mintoughnessthesis, Cao2025Structure, Katona2024Minimally, Ma2023Minimum, Ma2024Structure, Zheng2024Disproof}). In particular, Katona, Solt\'esz, and Varga showed that for every positive rational number~$t$, any graph is an induced subgraph of some minimally $t$-tough graph~\cite{article:min1toughgraphs}.
While Katona and Khan~\cite{Katona2024Minimally} characterized all minimally $t$-tough, chordal graphs with $t \le 1/2$, Katona and Varga~\cite{article:spec_graph_classes_journal} showed that for any $t \in (1/2,1]$, no chordal graph is minimally $t$-tough.

\begin{theorem}[Katona and Varga~\cite{article:spec_graph_classes_journal}] \label{thm:minttough_chordal_with_t_between_1/2_and_1}
 For any rational number $t \in (1/2,1]$, there exists no minimally $t$-tough, chordal graph.
\end{theorem}

In this paper, we continue the study of minimally $t$-tough graphs.
We pose the following conjecture.

\begin{conjecture}\label{conj:chordal}
For any rational number $t > 1/2$, there exists no minimally $t$-tough, chordal graph.
\end{conjecture}

First, we provide a necessary and sufficient condition for a graph $G$ to be minimally tough. 
We then use this condition to prove \cref{conj:chordal} for several subclasses of chordal graphs.
In particular, we show that for any $t>1/2$, no strongly chordal graph is minimally $t$-tough, no split graph is minimally \hbox{$t$-tough}, and no chordal graph with a universal vertex is minimally $t$-tough. 
The interesting property of these graph classes is that they are not closed under edge deletion.
We also show that for any $t\le 1$, the only minimally $t$-tough graphs with a universal vertex are stars.

In contrast to the aforementioned result, which states that no induced subgraph can be excluded from the class of minimally tough graphs~\cite{article:min1toughgraphs}, this paper addresses the question of which induced subgraphs, if any, must necessarily be present in each minimally tough graph. 
More precisely, our results imply that for any $t>1$, every minimally $t$-tough, noncomplete graph must satisfy the following:
\begin{itemize}[topsep=3pt, itemsep=0pt]
 \item[--] it must contain a hole or an induced subgraph isomorphic to the $k$-sun for some $k\ge 3$;
 \item[--] if it contains a universal vertex, then it must contain a hole;
 \item[--] it must contain an induced $4$-cycle, or an induced $5$-cycle, or two independent edges as an induced subgraph.
\end{itemize}
These results complement that of Katona and Varga stating that for any rational number $t \in (1/2,1]$, every minimally $t$-tough graph contains a hole~\cite{article:spec_graph_classes_journal}.

\section{Preliminaries} \label{preliminaries}

In this section, we present some necessary definitions and claims. Let $\omega(G)$ denote the \emph{number of components}\footnote{Using $\omega(G)$ to denote the number of components might be confusing; most of the literature on toughness, however, uses this notation.}, $\kappa(G)$ the \emph{connectivity} and $\delta(G)$ the \emph{minimum degree} of a graph $G$. For a vertex $v$ and for a set of vertices $W$ in a graph $G$, the \emph{degree} of $v$ is denoted by $d_G(v)$, and the \emph{open neighborhood} and the \emph{closed neighborhood} of $v$ and those of $W$ are denoted by $N_G(v)$, $N_G[v]$, $N_G(W)$, and $N_G[W]$, respectively. In all cases, the subscript $G$ is dismissed whenever it does not cause confusion.

\begin{definition}
 Let $t$ be a real number. A graph $G$ is called \emph{$t$-tough} if $|S| \ge t \cdot \omega(G-S)$ holds for any vertex set $S \subseteq V(G)$ that disconnects the graph (i.e., for any $S \subseteq V(G)$ with $\omega(G-S)>1$). The \emph{toughness} of $G$, denoted by $\tau(G)$, is the largest $t$ for which G is $t$-tough, taking $\tau(K_n) = \infty$ for all $n \ge 1$.
 
 A graph $G$ is said to be \emph{minimally $t$-tough} if $\tau(G) = t$ and $\tau(G - e) < t$ for all $e \in E(G)$. 
 A graph is called \emph{minimally tough} if it is minimally $t$-tough for some real number $t$.
\end{definition}

Note that a graph is disconnected if and only if its toughness is 0.

It is not difficult to see that the toughness of any connected, noncomplete graph is a positive rational number, thus there exist no minimally tough graphs with nonpositive or with irrational toughness. Also note that every complete graph on at least two vertices is minimally $\infty$-tough.

The following claim is an easy consequence of the definition of minimally tough graphs.

\begin{proposition}[Katona, Kov\'acs, and Varga~\cite{article:dp}] \label{claim:minttoughlemma}
 Let $t$ be a positive rational number and $G$ be a minimally $t$-tough graph. For every edge $e \in E(G)$,
 \begin{itemize}[topsep=3pt, itemsep=1pt]
  \item[--] the edge $e$ is a bridge in $G$, or
  \item[--] there exists a vertex set ${S=S(e) \subseteq V(G)}$ such that
   \[ \omega(G-S) \le \frac{|S|}{t} \quad \text{and} \quad \omega \big( (G-e)-S \big) > \frac{|S|}{t} \,, \]
   and the edge $e$ is a bridge in $G-S$.
 \end{itemize}
 In the first case, we define $S = S(e) = \emptyset$.
\end{proposition}

The following observation is straightforward.

\begin{observation}\label{bridge}
 Let $G$ be a graph, let $S\subseteq V(G)$, and let $e = uv$ be an edge of $G-S$ that is a bridge in $G-S$ (or, equivalently, $S$ is a $u$-$v$ separator in $G$).
 Then $\omega \big( (G-e)-S \big)  = \omega(G-S) + 1$.
\end{observation}

The following relation between the toughness of a graph and its connectivity can be proved directly from the definition of toughness.

\begin{proposition}[Chv\'{a}tal~\cite{article:toughness_introduction}] \label{claim:connectivity_and_toughness}
 For every noncomplete graph $G$, we have $\tau(G) \le \kappa(G)/2$.
\end{proposition}

Now we give some further definitions and theorems needed for this paper. Let $G$ be a graph. For two vertices $u,v \in V(G)$, a set $S \subseteq V(G) \setminus \{ u,v \}$ is called a \emph{$u$-$v$ separator} if $u$ and $v$ belong to different components of $G-S$. A $u$-$v$ separator is called \emph{minimal} if none of its proper subsets is a $u$-$v$ separator. 
A set $S \subseteq V(G)$ is called a \emph{(minimal) separator} if it is a (minimal) $u$-$v$ separator for some $u,v \in V(G)$. 
Observe that the empty set is a minimal separator if and only if the graph is disconnected.
Note that a minimal separator can be a proper subset of another minimal separator but for different pairs of vertices.

To give a characterization for minimal separators, we need the following definition.

\begin{definition}
 For a set $S \subseteq V(G)$, a component $C$ of $G-S$ is said to be \emph{$S$-full} if every vertex in $S$ has a neighbor in $C$.
\end{definition}

The concept of $S$-full components can be used to characterize minimal separators.

\begin{proposition}[Golumbic, \cite{book:minseparator_characterization}*{Exercise~10 of Chapter~4}] \label{claim:minseparator_characterization}
 In a graph $G$, a set $S \subseteq V(G)$ is a minimal separator if and only if the graph $G-S$ has at least two $S$-full components.
\end{proposition}

In this work, we mainly focus on chordal, strongly chordal, and split graphs. A \emph{hole} in a graph is an induced cycle of length at least 4. A graph is \emph{chordal} if it is hole-free, i.e., if it does not contain an induced cycle of length at least 4. A graph is \emph{strongly chordal} if it is chordal and every even cycle of length at least 6 has an odd chord, i.e., a chord whose endpoints are an odd distance apart in the cycle. A graph is an interval graph if each of its vertices can be assigned an interval on the real line so that two vertices are adjacent if and only if the corresponding intervals overlap. A graph is a \emph{split graph} if its vertex set can be partitioned into a clique and an independent set. It is not difficult to see that every split graph is chordal, but not necessarily strongly chordal.

A vertex or a set of vertices is called \emph{simplicial} if its closed neighborhood is a clique. 
A classical theorem of Dirac states that every chordal graph has a simplicial vertex; moreover, every noncomplete, chordal graph contains at least two nonadjacent simplicial vertices \cite{article:rigid_circuit}.

\begin{definition}\label{def:maximumNeighbor}
 A vertex $u \in N[v]$ is a \emph{maximum neighbor} of a vertex $v$ if $N[w] \subseteq N[u]$ holds for all $w \in N[v]$.
\end{definition}

Note that a vertex can be its own maximum neighbor.

\begin{definition}
 A vertex $s$ is called a \emph{simple vertex} if its closed neighborhood can be linearly ordered by inclusion, i.e., for any $x,y \in N[s]$, if $x$ precedes $y$, then $N[x] \subseteq N[y]$.
\end{definition}

Observe that every simple vertex is simplicial, and has a maximum neighbor.

We now introduce the notions of moplexes and moplicial vertices, which are crucial for our approach.

\begin{definition}
 In a graph $G$, a set of vertices $M \subseteq V(G)$ is a \emph{module} if each vertex $v \in V(G) \setminus M$ is either adjacent to every vertex in $M$ or not adjacent to any vertex in $M$.
 A \emph{clique module} is a module that is a clique.
 A \emph{moplex} in a graph $G$ is an inclusion-wise maximal clique module whose open neighborhood is either empty or a minimal separator in $G$.
 A vertex belonging to a moplex is said to be \emph{moplicial}.
\end{definition}

Let us first recall the following characterization of chordal graphs.

\begin{theorem}[Dirac~\cite{article:rigid_circuit}] \label{thm:in_chordal_graphs_moplicial_means_simplicial}
 A graph $G$ is chordal if and only if every minimal separator in $G$ forms a clique.
\end{theorem}

The following result is an immediate consequence of \cref{thm:in_chordal_graphs_moplicial_means_simplicial}.

\begin{corollary}\label{cor:moplicial-is-simplicial}
 If $G$ is a chordal graph, then every moplicial vertex in $G$ is simplicial.
\end{corollary}

\Cref{thm:in_chordal_graphs_moplicial_means_simplicial} was strengthened by Berry and Bordat~\cite{article:moplexes_in_graphs} using the concept of moplexes, which, in the case of chordal graphs, can consist of simplicial vertices only.

\begin{theorem}[Berry and Bordat~\cite{article:moplexes_in_graphs}] \label{thm:moplexes_in_graphs}
 Every noncomplete graph contains at least two moplexes. 
\end{theorem}

The following proposition provides a link between simple vertices and moplexes.

\begin{proposition} \label{lemma:simple_means_moplicial}
 If $s$ is a simple vertex of a graph $G$, then $s$ is moplicial.
\end{proposition}

\begin{proof}
 Let $M$ denote the set of vertices in $G$ whose closed neighborhood is the same as that of $s$. 
 Clearly, $M$ is an inclusion-wise maximal clique module. 
 If $N(M) = \emptyset$, then $M$ is a moplex and thus, $s$ is moplical.
 So we may assume that $N(M)\neq \emptyset$.
 In this case, we claim that $N(M)$ is a minimal separator in $G$.
 
 By \cref{claim:minseparator_characterization}, we need to show that $G-N(M)$ has at least two $N(M)$-full components.
 
 Since $M$ is a clique module, 
 $M$ forms an $N(M)$-full component in $G-N(M)$. 
 
 Let $v$ be a vertex in $N(M)$ with smallest degree (since $N(M) \ne \emptyset$, such a vertex exists). Since $s$ is simple and $v \notin M$, it follows that $N[s] \subsetneq N[v] \subseteq N[w]$ holds for any $w \in N(M)$. Take a vertex $u \in N[v] \setminus N[s]$, and let $H$ denote the connected component of $G-N(M)$ that contains $u$.
 Now we show that $H$ is an $N(M)$-full component in $G-N(M)$.
 Since $u \in N[v] \subseteq N[w]$ holds for any $w \in N(M)$, we can conclude that $H$ is also an $N(M)$-full component in $G-N(M)$. Therefore, $M$ is a moplex, thus $s$ is moplicial.
\end{proof}

Chordal graphs have several equivalent definitions; in this paper, we also rely on the one using the concept of clique trees.

\begin{definition} \label{def:clique_tree}
 Let $G$ be a graph and let $\mathcal{Q}_G$ denote the set of its maximal cliques. A \emph{clique tree} of $G$ is a tree with vertex set $\mathcal{Q}_G$ such that for every pair of distinct maximal cliques $Q, Q' \in \mathcal{Q}_G$, the set $Q \cap Q'$ is contained in every clique on the path connecting $Q$ and $Q'$ in the tree.
\end{definition}

The concept of clique trees was used independently by Gavril~\cite{article:intersection_graphs_of_subtrees} and by Blair and Payton~\cite{article:clique_trees} to characterize chordal graphs.

\begin{theorem}[Gavril~\cite{article:intersection_graphs_of_subtrees}, Blair and Payton~\cite{article:clique_trees}] \label{thm:clique_tree}
A connected graph is chordal if and only if it has a clique tree.
\end{theorem}

With the help of clique trees, minimal separators in chordal graphs can be characterized as follows.

\begin{proposition}[Blair and Payton~\cite{article:clique_trees}] \label{lemma:minseparators_in_chordal_graphs}
 Let $G$ be a connected, chordal graph and $T$ be a clique tree of $G$. A set of vertices $S \subseteq V(G)$ is a minimal separator in $G$ if and only if $S$ is the intersection of two maximal cliques corresponding to two adjacent vertices of $T$.
\end{proposition}

The following proposition shows an interesting connection between moplexes in a chordal graph and its clique trees.

\begin{proposition}[Berry and Bordat~\cite{article:moplexes_in_clique_trees}] \label{lemma:moplexes_in_clique_trees}
 For any moplex $M$ in a connected, noncomplete, chordal graph $G$, there exists a clique tree $T$ of $G$ such that $N[M]$ is a maximal clique corresponding to a leaf of $T$.
\end{proposition}

Building upon the notion of maximum neighbor (cf.\ \Cref{def:maximumNeighbor}), we introduce the following definition.

\begin{definition}\label{def:maximumNeighboringEdge}
 An edge $uu'$ is a \emph{maximum neighboring edge} of a vertex $v$ if $u, u' \in N(v)$ and $N[w] \subseteq N[u] \cup N[u']$ holds for all $w \in N[v]$.
\end{definition}

One particular way of obtaining a maximum neighboring edge is via maximum neighbors, as follows.

\begin{observation}\label{mne}
If a vertex $v$ has a maximum neighbor $u \ne v$, then for every neighbor $u' \ne u$ of $v$, the edge $uu'$ is a maximum neighboring edge of $v$.
\end{observation}

In \cref{sec:strongly_chordal}, we study strongly chordal graphs, where we rely on an equivalent characterization of these graphs. To present this characterization, first we need a definition.

\begin{definition}
 Given an integer $k \ge 3$, a \emph{$k$-sun} $S_k$ is a graph whose vertex set can be partitioned into two sets $A = \{ a_1, \ldots, a_k \}$ and $B = \{ b_1, \ldots, b_k \}$ such that $A$ is a clique and $B$ is an independent set in $S_k$, and for all $i,j \in \{ 1, \ldots, k \}$, the vertices $a_i$ and $b_j$ are adjacent if and only if $i=j$ or $i \equiv j+1 \pmod{k}$.
 A graph is said to be \emph{sun-free} if it does not contain an induced $k$-sun for any integer $k \ge 3$.
 \end{definition}

For an example of a $k$-sun, see \Cref{fig:k-sun}.
 
\begin{figure}[ht]
\centering
\begin{tikzpicture}[scale=0.9]
 \tikzstyle{vertex}=[draw,circle,fill=black,minimum size=6,inner sep=0]
 
 \node[vertex] (a1) at (90:1) {};
 \node[vertex] (a2) at (90+120:1) {};
 \node[vertex] (a3) at (90+2*120:1) {};
 
 \node[vertex] (b1) at (150:2) {};
 \node[vertex] (b2) at (150+120:2) {};
 \node[vertex] (b3) at (150+2*120:2) {};
 
 \draw[thick] (a1) -- (a2);
 \draw[thick] (a2) -- (a3);
 \draw[thick] (a3) -- (a1);
 
 \draw (a1) -- (b1) -- (a2) -- (b2) -- (a3) -- (b3) -- (a1);
 
 \node at (-1.75,1.5) {$S_3$};
 
 \begin{scope}[shift={(5,0)}]
 \node[vertex] (a1) at (45+0*90:1) {};
 \node[vertex] (a2) at (45+1*90:1) {};
 \node[vertex] (a3) at (45+2*90:1) {};
 \node[vertex] (a4) at (45+3*90:1) {};
 
 \node[vertex] (b1) at (90+0*90:{(sqrt(3)+1)/sqrt(2)}) {};
 \node[vertex] (b2) at (90+1*90:{(sqrt(3)+1)/sqrt(2)}) {};
 \node[vertex] (b3) at (90+2*90:{(sqrt(3)+1)/sqrt(2)}) {};
 \node[vertex] (b4) at (90+3*90:{(sqrt(3)+1)/sqrt(2)}) {};
 
 \draw (a1) -- (a2) -- (a3) -- (a4) -- (a1);
 \draw (a1) -- (a3);
 \draw (a2) -- (a4);
 
 \draw (a1) -- (b1) -- (a2) -- (b2) -- (a3) -- (b3) -- (a4) -- (b4) -- (a1);
 
 \node at (-1.75,1.5) {$S_4$};
 \end{scope}
 
 \begin{scope}[shift={(10,0)}]
 \node[vertex] (a1) at (90+0*72:1) {};
 \node[vertex] (a2) at (90+1*72:1) {};
 \node[vertex] (a3) at (90+2*72:1) {};
 \node[vertex] (a4) at (90+3*72:1) {};
 \node[vertex] (a5) at (90+4*72:1) {};
 
 \node[vertex] (b1) at (126+0*72:{cos(36) + sqrt(3)*sin(36)}) {};
 \node[vertex] (b2) at (126+1*72:{cos(36) + sqrt(3)*sin(36)}) {};
 \node[vertex] (b3) at (126+2*72:{cos(36) + sqrt(3)*sin(36)}) {};
 \node[vertex] (b4) at (126+3*72:{cos(36) + sqrt(3)*sin(36)}) {};
 \node[vertex] (b5) at (126+4*72:{cos(36) + sqrt(3)*sin(36)}) {};
 
 \draw (a1) -- (a2) -- (a3) -- (a4) -- (a5) -- (a1);
 \draw (a1) -- (a3) -- (a5) -- (a2) -- (a4) -- (a1);
 
 \draw (a1) -- (b1) -- (a2) -- (b2) -- (a3) -- (b3) -- (a4) -- (b4) -- (a5) -- (b5) -- (a1);
 
 \node at (-1.75,1.5) {$S_5$};
 \end{scope}
\end{tikzpicture}
\caption{The 3-sun, the 4-sun, and the 5-sun.}
\label{fig:k-sun}
\end{figure}

Now we are ready to present the equivalent characterizations of strongly chordal graphs.

\begin{theorem}[Farber~\cite{article:strongly_chordal_characterization}] \label{thm:strongly_chordal_characterization}
 For any graph $G$, the following conditions are equivalent.
 \begin{itemize}[topsep=3pt, itemsep=0pt]
  \item[--] The graph $G$ is strongly chordal.
  \item[--] The graph $G$ is chordal and sun-free.
  \item[--] Each induced subgraph of $G$ has a simple vertex.
 \end{itemize}
\end{theorem}

In \cref{sec:interval}, we study interval graphs and make use of the following equivalent characterization. Three independent vertices in a graph form an \emph{asteroidal triple} if every two of them are connected by a path avoiding the neighborhood of the third. A graph is called \emph{asteroidal-triple-free} if it does not contain an asteroidal triple.

\begin{theorem}[Lekkerkerker and Boland~\cite{article:interval_graphs}] \label{thm:interval}
 For any graph $G$, the following conditions are equivalent.
 \begin{itemize}[topsep=3pt, itemsep=0pt]
  \item[--] The graph $G$ is an interval graph.
  \item[--] The graph $G$ is chordal and asteroidal-triple-free.
  \item[--] The graph $G$ does not contain an induced subgraph isomorphic to one of the graphs in \Cref{fig:interval}.
 \end{itemize}
\end{theorem}

For the forbidden induced subgraphs in an interval graph, see \Cref{fig:interval}.

\begin{figure}[ht]
\centering
\begin{tikzpicture}
 \tikzstyle{vertex}=[draw,circle,fill=black,minimum size=6,inner sep=0]
 
 \node[vertex] (u1) at (3,3) [label=right:$u_1$] {};
 \node[vertex] (u2) at (0,3) [label=left:$u_2$] {};
 \node[vertex] (u3) at (0,0) [label=left:$u_3$] {};
 \node[vertex] (u4) at (3,0) [label=right:$u_k$] {};
 
 \draw[fill] (1.35,0) circle (0.015);
 \draw[fill] (1.5,0) circle (0.015);
 \draw[fill] (1.65,0) circle (0.015);
 
 \draw[thick] (u4) -- (u1) -- (u2) -- (u3);
 \draw[thick] (u3) -- (1,0);
 \draw[thick] (u4) -- (2,0);
 
 \begin{scope}[shift={(7,1)}]
 \node[vertex] (a1) at (0,0) {};
 \node[vertex] (a2) at (90:1) {};
 \node[vertex] (a3) at (90:2) {};
 \node[vertex] (a4) at (90+120:1) {};
 \node[vertex] (a5) at (90+120:2) {};
 \node[vertex] (a6) at (90+240:1) {};
 \node[vertex] (a7) at (90+240:2) {};
 
 \draw[thick] (a1) -- (a2) -- (a3);
 \draw[thick] (a1) -- (a4) -- (a5);
 \draw[thick] (a1) -- (a6) -- (a7);
 \end{scope}
 
 \begin{scope}[shift={(11,0)}]
 \node[vertex] (b1) at (1.5,3) {};
 \node[vertex] (b2) at (-0.5,1.5) {};
 \node[vertex] (b3) at (0.5,1.5) {};
 \node[vertex] (b4) at (1.5,1.5) {};
 \node[vertex] (b5) at (2.5,1.5) {};
 \node[vertex] (b6) at (3.5,1.5) {};
 \node[vertex] (b7) at (1.5,0) {};
 
 \draw[thick] (b1) -- (b2);
 \draw[thick] (b1) -- (b3);
 \draw[thick] (b1) -- (b4);
 \draw[thick] (b1) -- (b5);
 \draw[thick] (b1) -- (b6);
 \draw[thick] (b2) -- (b3) -- (b4) -- (b5) -- (b6);
 \draw[thick] (b4) -- (b7);
 \end{scope}

 \begin{scope}[shift={(3,-5)}]
 \node[vertex] (c1) at (0,3) {};
 \node[vertex] (c2) at (0,1.5) {};
 \node[vertex] (c3) at (-3,0) {};
 \node[vertex] (v1) at (-2,0) [label=below:$v_1$] {};
 \node[vertex] (v2) at (-0.75,0) [label=below:$v_2$] {};
 \node[vertex] (v3) at (0.75,0) [label=below:$v_{\ell-1}$] {};
 \node[vertex] (v4) at (2,0) [label=below:$v_{\ell}$] {};
 \node[vertex] (c4) at (3,0) {};
 
 \draw[thick] (c1) -- (c2);
 \draw[thick] (c3) -- (v1) -- (v2) -- (-0.325,0);
 \draw[thick] (0.325,0) -- (v3) -- (v4) -- (c4);
 \draw[thick] (c2) -- (v1);
 \draw[thick] (c2) -- (v2);
 \draw[thick] (c2) -- (v3);
 \draw[thick] (c2) -- (v4);
 
 \draw[fill] (-0.15,0) circle (0.015);
 \draw[fill] (0,0) circle (0.015);
 \draw[fill] (0.15,0) circle (0.015);
 
 \draw[fill] (-0.125,0.75) circle (0.015);
 \draw[fill] (0,0.75) circle (0.015);
 \draw[fill] (0.125,0.75) circle (0.015);
 \end{scope}
 
 \begin{scope}[shift={(11,-5)}]
 \node[vertex] (d1) at (0,3) {};
 \node[vertex] (d2) at (-1,1.5) {};
 \node[vertex] (d3) at (1,1.5) {};
 \node[vertex] (d4) at (-3,0) {};
 \node[vertex] (w1) at (-2,0) [label=below:$w_1$] {};
 \node[vertex] (w2) at (-0.75,0) [label=below:$w_2$] {};
 \node[vertex] (w3) at (0.75,0) [label=below:$w_{m-1}$] {};
 \node[vertex] (w4) at (2,0) [label=below:$w_{m}$] {};
 \node[vertex] (d5) at (3,0) {};
 
 \draw[thick] (d1) -- (d2);
 \draw[thick] (d1) -- (d3);
 \draw[thick] (d4) -- (w1) -- (w2) -- (-0.325,0);
 \draw[thick] (0.325,0) -- (w3) -- (w4) -- (d5);
 \draw[thick] (d2) -- (d3);
 \draw[thick] (d2) -- (d4);
 \draw[thick] (d2) -- (w1);
 \draw[thick] (d2) -- (w2);
 \draw[thick] (d2) -- (w3);
 \draw[thick] (d2) -- (w4);
 \draw[thick] (d3) -- (w1);
 \draw[thick] (d3) -- (w2);
 \draw[thick] (d3) -- (w3);
 \draw[thick] (d3) -- (w4);
 \draw[thick] (d3) -- (d5);
 
 \draw[fill] (-0.15,0) circle (0.015);
 \draw[fill] (0,0) circle (0.015);
 \draw[fill] (0.15,0) circle (0.015);
 
 \draw[fill] (-9/12,1) circle (0.015);
 \draw[fill] (-8/12,1) circle (0.015);
 \draw[fill] (-7/12,1) circle (0.015);
 
 \draw[fill] (7/12,1) circle (0.015);
 \draw[fill] (8/12,1) circle (0.015);
 \draw[fill] (9/12,1) circle (0.015);
 \end{scope}

\end{tikzpicture}
\caption{The forbidden induced subgraphs in an interval graph, where $k \ge 4$ and $\ell \ge 2$ and $m \ge 1$.}
\label{fig:interval}
\end{figure}

In \cref{sec:interval}, we study split graphs using the following characterization.

\begin{theorem}[F\"{o}ldes and Hammer~\cite{article:split_characterization}] \label{thm:split_characterization}
 A graph is a split graph if and only if it is $(C_4, C_5, 2K_2)$-free, i.e., if it does not contain a cycle of length 4, a cycle of length 5, and a pair of independent edges as induced subgraphs.
\end{theorem}

\section{A characterization of non-minimally tough graphs}

We first characterize those edges of a $t$-tough graph whose removal does not decrease the toughness.

\begin{lemma} \label{lem:not_minimally_tough}
 Let $t$ be a real number, $G$ be a $t$-tough graph, and let $e = uv$ be an edge in $G$.
 Then, $G-e$ is $t$-tough if and only if both of the following conditions are met.
 \begin{itemize}[topsep=4pt, itemsep=0pt]
  \item[--] There exist at least $2t+1$ pairwise internally vertex-disjoint $u$-$v$ paths in $G$ (including~$uv$).
  \item[--] Every separator $S$ in $G$ that is also a $u$-$v$ separator in $G-e$ satisfies 
   \[ |S| \ge t \cdot \big( \omega(G-S)+1 \big) . \]
 \end{itemize}
\end{lemma}

\begin{proof}
 First, assume that $G-e$ is $t$-tough.
 By \cref{claim:connectivity_and_toughness}, we have $\kappa(G-e) \ge 2 \tau(G-e) \ge 2t$, so by Menger's theorem, there exist at least $2t$ pairwise internally vertex-disjoint paths from $u$ to $v$ in $G-e$, which proves that the first condition holds.
 Consider now a separator $S$ in $G$ that is also a $u$-$v$ separator in $G-e$.
 Then $e$ is a bridge in $G-S$, which implies $\omega \big( (G-e) - S \big) = \omega(G-S)+1$.
 Since $S$ is a separator in $G-e$ and $G-e$ is $t$-tough, we obtain 
 \[ |S| \ge t \cdot \omega \big( (G-e) - S \big) =  t \cdot \big( \omega(G-S)+1 \big) , \]
 which proves that the second condition holds.
 
 \medskip
 
 Conversely, assume that the two conditions hold.
 Let $S$ be an arbitrary separator in $G-e$. 
 We need to show that $|S| \ge t \cdot \omega \big( (G-e) - S \big)$ holds.
 
 Assume first that $S$ is not a separator in $G$. 
 Then $\omega(G-S) = 1$ and $e$ is a bridge in $G-S$, which implies
 \[ \omega \big( (G-e) - S \big) = \omega(G-S) + 1 = 2 . \]
 Thus, $S$ is a $u$-$v$ separator in $G-e$ and so, by the first condition, we have $|S| \ge 2t$. Therefore,
 \[ |S| \ge 2t = t \cdot \omega \big( (G-e) - S \big) . \]
 
 Assume now that $S$ is a separator in $G$, as well as a $u$-$v$ separator in $G-e$. Then, by \cref{bridge}, we have $\omega \big( (G-e) - S \big) = \omega(G-S) + 1$. So by the second condition, we have
 \[ |S| \ge t \cdot \big( \omega(G-S)+1 \big) = t \cdot \omega \big( (G-e) - S \big) . \]
 
 Finally, assume that $S$ is a separator in $G$, but not a $u$-$v$ separator in $G-e$. Then $\omega \big( (G-e) - S \big) = \omega(G-S)$ and using the fact that $G$ is $t$-tough, we obtain 
 \[ |S| \ge t \cdot \omega(G-S) = t \cdot \omega \big( (G-e) - S \big) . \]
 This completes the proof.
\end{proof}

Note that a graph $G$ with $t = \tau(G)$ is not minimally $t$-tough if and only if it contains an edge $e$ such that $G-e$ is $t$-tough. 
Therefore, \Cref{lem:not_minimally_tough} implies the following characterization of not minimally tough graphs.

\begin{theorem} \label{thm:not_minimally_tough}
 Let $t$ be a rational number and let $G$ be a graph with $\tau(G) = t$. 
 Then, $G$ is not minimally $t$-tough if and only if $G$ contains an edge $e = uv$ such that both of the following conditions are met.
 \begin{enumerate}[topsep=3pt, itemsep=0pt, label=(\arabic*)]
  \item \label[condition]{cond:Menger-thm} There exist at least $2t+1$ pairwise internally vertex-disjoint $u$-$v$ paths in $G$ (including~$uv$).
  \item \label[condition]{cond:separators_are_large-thm} Every separator $S$ in $G$ that is also a $u$-$v$ separator in $G-e$ satisfies 
   \[ |S| \ge t \cdot \big( \omega(G-S)+1 \big) . \]
 \end{enumerate}
\end{theorem}

As a side remark, let us note that \cref{lem:not_minimally_tough} can be further strengthened by the use of the following lemma.

\begin{lemma} \label{lemma:technical_lemma_about_separators}
 Let $t$ be a real number, $G$ be a graph, and let $e = uv$ an edge in $G$.
 Then, the following two conditions are equivalent.
 \begin{enumerate}[topsep=3pt, itemsep=0pt, label=(\alph*)]
  \item \label[condition]{cond:separator_simple_version} 
  Every separator $S$ in $G$ that is also a $u$-$v$ separator in $G-e$ satisfies: 
  \[ |S| \ge t \cdot \big( \omega(G-S)+1 \big) . \]
  \item \label[condition]{cond:separator_restricted_version} Every separator $S$ in $G$ that is also a $u$-$v$ separator in $G-e$, and in which every vertex has neighbors in at least two components of ${(G-e)-S}$, satisfies: 
  \[ |S| \ge t \cdot \big( \omega(G-S)+1 \big) . \]
 \end{enumerate}
\end{lemma}

\begin{proof}
 \Cref{cond:separator_simple_version} trivially implies \cref{cond:separator_restricted_version}.

 To prove the converse, assume that \cref{cond:separator_restricted_version} holds. We show that every separator $S$ in $G$ that is also a $u$-$v$ separator in $G-e$ satisfies $|S| \ge t \cdot \big( \omega(G-S)+1 \big)$ by induction on $k = k(S)$, defined as the number of vertices in $S$ that have neighbors in at most one component of $(G-e)-S$.
 
 The base case corresponds to $k = 0$, where the desired inequality holds by \cref{cond:separator_restricted_version}. Let $S$ be an arbitrary separator in $G$ that is also a $u$-$v$ separator in $G-e$ with $k = k(S) > 0$, and assume that $|S'| \ge t \cdot \big( \omega(G-S')+1 \big)$ holds for any separator $S'$ in $G$ that is also a $u$-$v$ separator in $G-e$ with $k(S') < k$. Let $w \in S$ be a vertex that has neighbors in at most one component of $(G-e)-S$ and let $S' = S \setminus \{w\}$. It is not difficult to see that $S'$ is a separator in $G$ and also a $u$-$v$ separator in $G-e$, and $\omega \big( (G-e) - S' \big) = \omega \big( (G-e) - S \big)$, and $k(S') < k(S)$. Thus, by the induction hypothesis and by \cref{bridge},
 \[ |S| > |S'| \ge t \cdot \big( \omega(G-S')+1 \big) = t \cdot \omega \big( (G-e)-S' \big) = t \cdot \omega \big( (G-e) - S \big) = t \cdot \big( \omega(G-S)+1 \big) , \]
which completes the proof.
\end{proof}

Using \cref{thm:not_minimally_tough}, we now derive the following sufficient condition for a graph not to be minimally tough.

\begin{lemma} \label{lemma:sufficient_condition_for_not_minttough}
 Let $t$ be a positive rational number and let $G$ be a graph containing two adjacent vertices $u$ and $v$ such that $u$ and $v$ have at least $2t$ common neighbors, at least $t$ of which have all their neighbors in $N(u) \cup N(v)$. Then $G$ is not minimally $t$-tough.
\end{lemma}

\begin{proof}
 We can assume that $\tau(G) = t$, otherwise $G$ is clearly not minimally $t$-tough. 
 We verify that the graph $G$ and its edge $e=uv$ satisfy \cref{cond:Menger-thm,cond:separators_are_large-thm} of \cref{thm:not_minimally_tough}. 
 Clearly, \cref{cond:Menger-thm} follows from the assumption that $u$ and $v$ are adjacent and have at least $2t$ common neighbors. 
 For \cref{cond:separators_are_large-thm}, let $S$ be an arbitrary separator in $G$ that is also a $u$-$v$ separator in $G-e$. 
 Let
 \[ W = \big\{ w \in N(u) \cap N(v) \bigm| N(w) \subseteq N(u) \cup N(v) \big\} \,. \]
 By the assumption of the lemma, $|W| \ge t$ holds. 
 Furthermore, $W \subseteq S$ since $S$ is a $u$-$v$ separator in $G-e$. 
 Consider the vertex set $S' = S \setminus W$. 
 Since $u$ and $v$ are adjacent in $G$ and $S$ is a $u$-$v$ separator in $G-e$, the vertices $u$ and $v$ must belong to the same component of $G-S$; let us denote this component by $C$. Since $N(u) \cup N(v) \subseteq V(C) \cup S$, it follows that for every vertex $w\in W$, we have $N(w) \subseteq V(C) \cup W \cup S'$. Thus $S'$ is a separator in $G$ with $\omega(G-S') = \omega(G-S)$.
 Since $G$ is $t$-tough, we obtain
 \[ |S| = |S'| + |W| \ge t \cdot \omega(G-S') + t = t \cdot \big( \omega(G-S') + 1 \big) = t \cdot \big( \omega(G-S) + 1 \big) . \]
 Therefore, \cref{cond:separators_are_large-thm} holds as well, and by \cref{thm:not_minimally_tough}, we conclude that $G$ is not minimally $t$-tough.
\end{proof}

\section{Chordal graphs}

In this section, we present several applications of \cref{lemma:sufficient_condition_for_not_minttough} to derive results on minimally tough, chordal graphs.
First of all, let us note that \cref{lemma:sufficient_condition_for_not_minttough} can be used to provide an alternative proof of \cref{thm:minttough_chordal_with_t_between_1/2_and_1}; for this alternative proof, see the \hyperlink{sec:appendix}{Appendix}.

By \cref{thm:moplexes_in_graphs}, every graph contains moplicial vertices. 
In the following lemma, we study these vertices in minimally $t$-tough, chordal graphs with $t > 1/2$.

\begin{lemma} \label{thm:no_mintough_chordal_graph_with_special_vertices}
 Let $t > 1/2$ be a rational number and let $G$ be a minimally $t$-tough, chordal graph. 
 Then none of the moplicial vertices of $G$ have a maximum neighbor or a maximum neighboring edge.
\end{lemma}

\begin{proof}
 Suppose to the contrary that $G$ is a minimally $t$-tough, chordal graph containing a moplicial vertex $s$ with a maximum neighbor or with a maximum neighboring edge. First, we show that in both cases, $s$ has a maximum neighboring edge that satisfies some additional properties.
 
 Since $t \ne 0$ and $t \ne \infty$, the graph $G$ is connected and noncomplete. By \cref{thm:minttough_chordal_with_t_between_1/2_and_1}, we have $t > 1$.
 
 By \Cref{cor:moplicial-is-simplicial}, the vertex $s$ is simplicial.
 Since $s$ is moplicial, there exists a moplex $X$ containing $s$. 
 Since $s$ is simplicial, $N[s] = N[X]$ is the unique maximal clique containing $s$. 
 Let $Q = N[X]$. 
 By \cref{lemma:moplexes_in_clique_trees}, since $G$ is connected and noncomplete, $Q$ corresponds to a leaf of some clique tree $T$ of $G$. 
 Hence, there is a maximal clique $Q'$ of $G$ that corresponds to the only neighbor of this leaf in $T$.
 By \cref{lemma:minseparators_in_chordal_graphs}, we have that $Q \cap Q'$ is a minimal separator in $G$, so by \cref{claim:connectivity_and_toughness}, we have $|Q \cap Q'| \ge 2t$. 
 Since $Q \ne Q'$, it follows that $|Q| \ge |Q \cap Q'| + 1 \ge 2t + 1$.

 Note that \cref{claim:connectivity_and_toughness} implies that $s$ has at least two neighbors.
 Since $G$ is connected and noncomplete, $s$ cannot be its own maximum neighbor.
 Therefore, by \Cref{mne}, if $s$ has a maximum neighbor $x$, then for any other neighbor $x'$ of $s$, the edge $xx'$ is a maximum neighboring edge of $s$. 
 Consequently, $s$ always has a maximum neighboring edge.
 Furthermore, since $|Q \cap Q'| \ge 2t > 2$ and since for any $y\in Q\setminus Q'$ and $y'\in Q\cap Q'$, we have $N[y]\subsetneq N[y']$, we can conclude that $s$ has a maximum neighboring edge $uv$ such that $\{u,v\}\subseteq Q \cap Q'$.

 Therefore, we have $(Q\cup Q')\setminus \{u,v\}\subseteq N(u)\cap N(v)$. 
 This implies that the number of common neighbors of $u$ and $v$ is at least $|Q\cup Q'|-2\ge (|Q|+1)-2\ge 2t$.
 
 Now consider the vertices of $Q \setminus \{ u, v \}$. 
 Clearly, they are common neighbors of $u$ and $v$, and since $Q = N[s]$ and $uv$ is a maximum neighboring edge of $s$, the vertices of $Q \setminus \{ u, v \}$ have all their neighbors in $N[u] \cup N[v]$. 
 Furthermore,
 \[ \big| Q \setminus \{ u, v \} \big| = |Q| - 2 \ge (2t+1) - 2 = 2t - 1 > t \,, \] where the last inequality is valid since $t > 1$. 
 Therefore, \cref{lemma:sufficient_condition_for_not_minttough} implies that $G$ is not minimally $t$-tough, which is a contradiction.
\end{proof}

\subsection{Strongly chordal graphs} \label{sec:strongly_chordal}

\cref{thm:no_mintough_chordal_graph_with_special_vertices} has several interesting consequences.
The first one is related to strongly chordal graphs and relies on \cref{thm:strongly_chordal_characterization}.

\begin{theorem} \label{corollary:strongly_chordal}
 For any rational number $t > 1/2$, there exists no minimally $t$-tough, strongly chordal graph.
\end{theorem}
\begin{proof}
 Let $t > 1/2$ be a rational number and let $G$ be a strongly chordal graph. By \cref{thm:strongly_chordal_characterization}, the graph $G$ contains a simple vertex $s$, and by \cref{lemma:simple_means_moplicial}, the vertex $s$ is moplicial. 
 By a previous observation, since every simple vertex has a maximum neighbor, \cref{thm:no_mintough_chordal_graph_with_special_vertices} implies that $G$ is not minimally $t$-tough.
\end{proof}

Thus, by \cref{thm:strongly_chordal_characterization}, we can conclude the following.

\begin{corollary}
 For any rational number $t > 1/2$, every minimally $t$-tough graph contains  a hole or an induced $k$-sun for some integer $k \ge 3$.
\end{corollary}

\subsection{Interval graphs} \label{sec:interval}

Since every interval graph is strongly chordal~\cite{article:interval_graphs}, \cref{corollary:strongly_chordal} directly implies the following.

\begin{corollary} \label{corollary:interval}
 For any rational number $t > 1/2$, there exists no minimally $t$-tough, interval graph.
\end{corollary}

Using \cref{thm:interval}, we can equivalently restate \cref{corollary:interval} as follows.

\begin{corollary}
 For any rational number $t > 1/2$, every minimally $t$-tough graph contains an induced subgraph isomorphic to one of the graphs in \Cref{fig:interval}.
\end{corollary}

Let us remark that the result of \Cref{corollary:interval} cannot be generalized to the entire class of asteroidal-triple-free graphs; in fact, not even to the class of co-comparability graphs (which is a subclass of asteroidal-triple-free graphs that properly contains the class of interval graphs).
A graph is a \emph{co-comparability graph} if its complement admits a transitive orientation, that is, if $(u,v)$ and $(v,w)$ are arcs of the orientation, then $(u,w)$ must also be an arc of it.
For an integer $k \ge 2$, consider the graph $G$ consisting of two disjoint cliques of size $k$ and a perfect matching between them.
Then $G$ is a co-bipartite graph (i.e., the complement of a bipartite graph), and thus a co-comparability graph. Furthermore, $G$ is claw-free and $\kappa(G) = k$ since we need to remove at least one vertex from each edge of the perfect matching to disconnect $G$ and $G$ is $k$-regular. 
By a theorem of Matthews and Sumner, see \cite{article:clawfree}*{Theorem 10}, the toughness of any noncomplete claw-free graph is equal to half of its connectivity; hence, $\tau(G) = \kappa(G)/2 = k/2$. 
Since $G$ is $k$-regular, \Cref{claim:connectivity_and_toughness} implies $\tau(G-e) \le \kappa(G-e)/2 \le (k-1)/2$. Therefore, $G$ is a minimally $k/2$-tough, co-comparability graph.

\subsection{Chordal graphs with a universal vertex}

We now move to the study of minimally tough graphs with a \emph{universal vertex}, that is, a vertex adjacent to all the other vertices. 
Using \cref{claim:minttoughlemma,claim:connectivity_and_toughness}, we can derive a complete characterization of minimally tough graphs with a universal vertex and with toughness not exceeding~$1$.

\medskip

\begin{theorem}\label{thm:stars}
 Let $t \le 1$ be a rational number. 
 If $G$ is a minimally $t$-tough graph with a universal vertex, then $G \cong K_{1,\ell}$ and $t=1/\ell$ for some integer $\ell \ge 2$.
\end{theorem}

\begin{proof}
 Let $G$ be a minimally $t$-tough graph, and let $u \in V(G)$ be a universal vertex of $G$. Take a vertex $v \in V(G)$ distinct from $u$. Since $u$ is universal, we have $uv \in E(G)$. Let $S = S(uv) \subseteq V(G)$ be a vertex set guaranteed by \cref{claim:minttoughlemma}. Since $uv$ is a bridge in $G-S$, clearly $u \notin S$. 
 Since $u$ is universal in $G$, the graph $(G-uv)-S$ has exactly two components: the one containing only $v$ and the other containing all the remaining vertices. Therefore, $N(v)\setminus\{ u \} \subseteq S$. 
 Then by \cref{claim:minttoughlemma}, we have
 \[ \frac{\big| N(v) \big| - 1}{t} = \frac{\big| N(v)\setminus\{ u \} \big|}{t} \le \frac{|S|}{t} < \omega \big( (G-uv) - S \big) = 2 , \]
 that is,
 \[ \big| N(v) \big| < 2t+1 . \]
 Note that the above arguments and inequalities also hold in the case when $uv$ is a bridge in $G$, in which case $S = \emptyset$.
 
By \cref{claim:connectivity_and_toughness}, since $\kappa(G)$ is an integer, we have that $\lceil 2t \rceil \le\kappa(G) \le d(v) < 2t+1$ holds for any vertex $v \ne u$; moreover, since $d(v)$ is an integer, we have $d(v) = \lceil 2t \rceil$. This means, that if $t \le 1/2$, then each vertex, except possibly for $u$, is of degree 1. Thus $G$ is a star, that is, $G \cong K_{1,\ell}$ for some integer $\ell \ge 0$, and since the toughness of $K_1$ and $K_2$ is infinity,  we have $G \cong K_{1,\ell}$ and $t=1/\ell$ for some integer $\ell \ge 2$. If $1/2 < t \le 1$, then each vertex, except for $u$, is of degree 2, thus $G$ can be obtained from a star on at least 3 vertices by adding a matching covering every leaf of the star. 
However, no such graph is minimally $t$-tough with $1/2 < t \le 1$: if $G$ has exactly 3 vertices, then $G \cong K_3$, whose toughness is infinity, and if $G$ contains more than 3 vertices, then $u$ is a cut-vertex in $G$, thus the toughness of $G$ is at most $1/2$.
\end{proof}

For the case of toughness $t>1$, a complete characterization of minimally $t$-tough graphs with a universal vertex remains open. 
We provide a partial result in this direction, by identifying an infinite family of such graphs 
and
showing that every minimally tough graph with a universal vertex and with toughness exceeding $1$ contains one of these graphs as an induced subgraph.
This family of graphs is the \emph{wheel graphs}, defined as the class of graphs obtained from a cycle on at least $4$ vertices by adding a universal vertex.
The wheel graph on $n$ vertices is denoted by $W_n$; for some examples, see \Cref{fig:wheel}.

\begin{figure}[ht]
\centering
\begin{tikzpicture}[scale=0.9]
 \tikzstyle{vertex}=[draw,circle,fill=black,minimum size=6,inner sep=0]
 
 \node[vertex] (a1) at (90+0*90:1.5) {};
 \node[vertex] (a2) at (90+1*90:1.5) {};
 \node[vertex] (a3) at (90+2*90:1.5) {};
 \node[vertex] (a4) at (90+3*90:1.5) {};
 
 \node[vertex] (b) at (0,0) {};
 
 \draw (a1) -- (a2) -- (a3) -- (a4) -- (a1);
 \draw (b) -- (a1);
 \draw (b) -- (a2);
 \draw (b) -- (a3);
 \draw (b) -- (a4);
 
 \node at (-1.75,1.5) {$W_5$};
 
 \begin{scope}[shift={(7.5,0)}]
 \node[vertex] (a1) at (90+0*72:1.5) {};
 \node[vertex] (a2) at (90+1*72:1.5) {};
 \node[vertex] (a3) at (90+2*72:1.5) {};
 \node[vertex] (a4) at (90+3*72:1.5) {};
 \node[vertex] (a5) at (90+4*72:1.5) {};
 
 \node[vertex] (b) at (0,0) {};
 
 \draw (a1) -- (a2) -- (a3) -- (a4) -- (a5) -- (a1);
 \draw (b) -- (a1);
 \draw (b) -- (a2);
 \draw (b) -- (a3);
 \draw (b) -- (a4);
 \draw (b) -- (a5);
 
 \node at (-1.75,1.5) {$W_6$};
 \end{scope}
\end{tikzpicture}
\caption{The wheels on 5 and 6 vertices.}
\label{fig:wheel}
\end{figure}

We first show that wheel graphs are indeed minimally tough.

\begin{proposition}
 For any $n \ge 5$, the wheel graph $W_n$ on $n$ vertices is a minimally $t_n$-tough graph with a universal vertex, where $t_n =1+ 2/(n-1)$ when $n$ is odd and $t_n = 1+2/(n-2)$ when $n$ is even.
\end{proposition}

\begin{proof}
 First, we prove that $\tau(W_n) \ge t_n$ by showing that $|S|/\omega(G-S) \ge t_n$ holds for any cutset $S$ of $W_n$. Let $S$ be an arbitrary cutset of $W_n$. Since $S$ is a cutset, it must contain the universal vertex. Note that after the removal of the universal vertex from $W_n$, the graph becomes a cycle; hence $\omega(W_n - S) \le |S| - 1$. Also note that picking an arbitrary vertex from each component of $W_n-S$ forms an independent set of $W_n$, which implies $\omega(W_n - S) \le \lfloor (n-1)/2 \rfloor$. Therefore,
 \[ \frac{|S|}{\omega(W_n - S)} \ge \frac{\omega(W_n - S) + 1}{\omega(W_n - S)} \ge 1 + \frac{1}{\omega(W_n - S)} \ge 1 + \frac{1}{\lfloor (n-1)/2 \rfloor} = t_n. \]
 Thus, $\tau(W_n) \ge t_n$.
 
 Now let $w$ be the universal vertex of $W_n$ (that is, the center of the wheel), let $S'$ be a maximum-size independent set of $W_n$, and let $S = S' \cup \{ w \}$. Then
 \[ \tau(W_n) \le \frac{|S|}{\omega(W_n - S)} = \frac{\lfloor (n-1)/2 \rfloor + 1}{\lfloor (n-1)/2 \rfloor} = 1 + \frac{1}{\lfloor (n-1)/2 \rfloor} = t_n. \]

 Hence, $\tau(W_n) = t_n$.
	
 Finally, we verify that $\tau(W_n - e) < t_n$ for any edge $e$ of $W_n$. Observe that removing any edge $e$ from $W_n$ creates a vertex of degree 2, implying $\tau(W_n - e) \le 1 < t_n$.

 Therefore, $W_n$ is minimally $t_n$-tough.
\end{proof}

A graph $G$ is said to be \emph{wheel-free} if it does not contain any induced subgraph isomorphic to a wheel $W_n$ for some $n\ge 5$.
Now we show that for any rational number $t > 1$, there exists no minimally $t$-tough, wheel-free graph with a universal vertex.
Since a wheel-free graph can only have a universal vertex if it is chordal, it suffices to prove the following.

\begin{theorem} \label{thm:chordal_universal}
 For any rational number $t > 1$, there exists no minimally $t$-tough, chordal graph with a universal vertex.
\end{theorem}

\begin{proof}
 Let $t > 1$ be a rational number and let $G$ be a chordal graph with a universal vertex. We need to show that $G$ is not minimally $t$-tough. We may assume that $G$ is noncomplete; otherwise $G$ is clearly not minimally $t$-tough. 
 Let $U$ be the set of universal vertices in $G$ and let $G' = G-U$. Since $G$ is noncomplete, $G'$ has vertices and $G'$ is also noncomplete.
 
 Thus, by \cref{thm:moplexes_in_graphs}, the graph $G'$ contains a moplex $M$. Now we show that $M$ is also a moplex in $G$. Since $M$ is an inclusion-wise maximal clique module in $G'$ and $U$ is adjacent to every vertex in $G$, we infer that $M$ is also an inclusion-wise maximal clique module in $G$. Since $M$ is a moplex in the noncomplete graph $G'$, it follows that $N_{G'}(M)$ is a minimal separator in $G'$ (even if $N_{G'}(M) = \emptyset$, in which case $G'$ is disconnected). Furthermore, this implies that $N_{G}(M) = N_{G'}(M) \cup U$ is a minimal separator in $G$. Thus $M$ is indeed a moplex in $G$.
 
 Let $s \in M$ and $u \in U$ be arbitrary. Clearly, $s$ is moplicial and $u$ is a maximum neighbor of $s$, so by \cref{thm:no_mintough_chordal_graph_with_special_vertices}, the graph $G$ is not minimally $t$-tough.
\end{proof}

Since a graph with a universal vertex is chordal if and only if it is wheel-free, as an immediate consequence of \Cref{thm:stars,thm:chordal_universal}, we obtain the following.

\begin{corollary} \label{cor:wheel-free_universal}
For any rational number $t > 1/2$, there exists no minimally $t$-tough, wheel-free graph with a universal vertex.
\end{corollary}

In particular, for any rational number $t > 1/2$, there exists no minimally $t$-tough, chordal graph with a universal vertex.

\subsection{Split graphs}\label{sec:split}

Finally, we turn our attention to minimally tough, split graphs.
Katona and Varga gave a complete characterization of minimally tough, split graphs~\cite{article:spec_graph_classes_journal}. In particular, they showed that no minimally $t$-tough, split graph exists for any $t > 1/2$.
This latter result can also be derived using our findings regarding moplexes; for this alternative proof, see the \hyperlink{sec:appendix}{Appendix}.

Let us also note that using \cref{thm:split_characterization}, the theorem on the non-existence of minimally $t$-tough, split graphs with $t > 1/2$ can be equivalently stated as follows.

\begin{corollary}
 For any rational number $t > 1/2$, every minimally $t$-tough graph contains either an induced 4-cycle or an induced 5-cycle or a pair of independent edges as an induced subgraph.
\end{corollary}

In conclusion, let us note that if \Cref{conj:chordal} were true, then it would provide a common generalization of \crefnosort{thm:minttough_chordal_with_t_between_1/2_and_1,corollary:strongly_chordal,thm:chordal_universal,thm:split_graphs}.

\section*{Acknowledgments}

The authors are grateful to the anonymous reviewers for their helpful comments.
The research of Blas Fern\'andez is supported in part by the Slovenian Research and Innovation Agency (research program P1-0285 and research project J1-50000).
The research of Gyula Y.~Katona is supported by the BME-Artificial Intelligence FIKP grant of EMMI (BME FIKP-MI/SC), and the research  project K-132696 by the Ministry of Innovation and Technology of Hungary from the National Research, Development and Innovation Fund.
The research of Martin Milanič is supported in part by the Slovenian Research and Innovation Agency (I0-0035, research program P1-0285 and research projects J1-3003, J1-4008, J1-4084, J1-60012, and N1-0370) and by the research program CogniCom (0013103) at the University of Primorska.
The research of Kitti Varga is supported by the Ministry of Innovation and Technology of Hungary from the National Research, Development and Innovation Fund -- grant number ADVANCED 150556.

\clearpage

\appendix

\section{Appendix} \linkdest{sec:appendix}

To demonstrate further applications of our results, we present alternative proofs for two known theorems on the non-existence of minimally $t$-tough, chordal graphs for $1/2 < t \le 1$, and on the non-existence of minimally $t$-tough, split graphs for $t > 1/2$. Let us note that these proofs differ substantially from those given in~\cite{article:spec_graph_classes_journal}.

\subsection{Alternative proof of Theorem~\ref{thm:minttough_chordal_with_t_between_1/2_and_1}}

We start with re-proving the non-existence of minimally $t$-tough, chordal graphs for $1/2 < t \le 1$.

\begin{proof}[Proof of \cref{thm:minttough_chordal_with_t_between_1/2_and_1}]
 Let $t \in (1/2,1]$ be a rational number and let $G$ be a chordal graph with $\tau(G) = t$. Now we prove that $G$ is not minimally $t$-tough.
 
 Since $\tau(G) = t$, the graph $G$ is connected and noncomplete. By \cref{thm:clique_tree}, since $G$ is a connected chordal graph, it has a clique tree $T$. Since $G$ is noncomplete, $T$ has at least two nodes.
 
 Let $Q$ be a maximal clique in $G$ corresponding to a leaf of $T$, and let $Q'$ be the maximal clique of $G$ corresponding to the unique neighbor of this leaf in $T$. By \cref{lemma:minseparators_in_chordal_graphs}, we have that $S = Q \cap Q'$ is a minimal separator in $G$. By \cref{claim:connectivity_and_toughness}, we have $|S| \ge 2t > 1$, i.e., $|S| \ge 2$, thus $S$ contains two adjacent vertices $u$ and $v$.
 
 Since $Q$ and $Q'$ are distinct maximal cliques in $G$, there exist a vertex $x \in Q \setminus Q'$ and a vertex $y \in Q' \setminus Q$. Clearly, the vertices $x$ and $y$ are common neighbors of $u$ and $v$, thus $u$ and $v$ have at least $2 \ge 2t$ common neighbors. Now, we show that at least $1 \ge t$ of the common neighbors of $u$ and $v$, namely $x$, have all their neighbors in $N(u) \cup N(v)$. For this, observe that the set of vertices of $T$ that correspond to maximal cliques containing any fixed vertex in $G$ induces a subtree of $T$. Since $x \in Q \setminus Q'$ and $Q$ corresponds to a leaf in $T$, this implies that $Q$ is the only maximal clique containing $x$ in $G$. Therefore $N(x) = V(Q) \subseteq N(u) \cup N(v)$. So by \cref{lemma:sufficient_condition_for_not_minttough}, we infer that $G$ is not minimally $t$-tough.
\end{proof}

\subsection{Alternative proof of the non-existence of minimally \texorpdfstring{$\boldsymbol{t}$}{}-tough, split graphs for \texorpdfstring{$\boldsymbol{t > 1/2}$}{}}

In this section, we re-prove the following theorem.

\begin{theorem}[Katona and Varga~\cite{article:spec_graph_classes_journal}] \label{thm:split_graphs}
 For any rational number $t > 1/2$, there exists no minimally $t$-tough, split graph.
\end{theorem}

Our proof relies on a lemma of Katona and Varga~\cite{article:spec_graph_classes_journal}.

\begin{lemma}[Katona and Varga~\cite{article:spec_graph_classes_journal}] \label{claim:split_mintoughset}
 Let $t$ be a positive rational number and $G$ a minimally $t$-tough, split graph partitioned into a clique $Q$ and an independent set $I$. Let $e = uv$ be an edge between two vertices of $Q$ and $S = S(e) \subseteq V(G)$ a vertex set guaranteed by \cref{claim:minttoughlemma}. Then 
 \[ S = \big( Q \setminus \{ u, v \} \big) \cup \{ w \in I \mid uw,vw \in E(G) \} . \]
\end{lemma}

We are now ready to present our alternative proof of \cref{thm:split_graphs}, using the concept of moplexes.

\begin{proof}[Proof of \cref{thm:split_graphs}]
 Suppose to the contrary that there exists a minimally $t$-tough, split graph $G$ for some rational number $t>1/2$. Since every split graph is chordal, we can assume by \cref{thm:minttough_chordal_with_t_between_1/2_and_1} that $t > 1$.
 Then $G$ is connected and noncomplete, otherwise its toughness would be either 0 or infinity, and not $t$.
 Since $G$ is a split graph, its vertex set can be partitioned into a clique $Q$ and an independent set $I$.
 We can assume that in this partition, $Q$ is inclusion-wise maximal, i.e., none of the vertices of $I$ is adjacent to every vertex of $Q$. 
 
 \begin{claim} \label{degree3}
  Every vertex of $G$ has degree at least $3$.
 \end{claim}
 \begin{proof}
  By \cref{claim:connectivity_and_toughness}, every vertex of $G$ has degree at least $\kappa(G)\ge \lceil 2t \rceil \ge 3$, where the last inequality holds since $t > 1$.
 \end{proof}

 \begin{claim} \label{I2}
  We have $|I| \ge 2$.
 \end{claim}
 \begin{proof}
  Since $G$ is noncomplete, we have $I\neq \emptyset$.
  Furthermore, if $|I|=1$, then the neighbors of the only vertex in $I$ would be universal (and by \cref{degree3}, the only vertex in $I$ indeed has neighbors), contradicting \cref{thm:chordal_universal}.
 \end{proof}

 \begin{claim} \label{Q3}
  We have $|Q| \ge 3$.
 \end{claim}
 \begin{proof}
  By~\cref{I2}, there exists a vertex $v\in I$.
  By~\cref{degree3} and by $N(v)\subseteq Q$, we infer that $|Q|\ge 3$.
 \end{proof}

 \begin{claim} \label{Imoplex}
  Each vertex of $I$ forms a moplex.
 \end{claim}
 \begin{proof}
  Let $v \in I$ be an arbitrary vertex. Obviously, $\{ v \}$ is a clique module, and now we show that it is an inclusion-wise maximal one. So suppose to the contrary that there exists a clique module $M$ such that $v \in M$ and $M \ne \{ v \}$. Since $M$ is a clique and none of the vertices of $I$ is adjacent to every vertex of $Q$, clearly $M \setminus \{ v \} \subseteq N(v) \subsetneq Q$ holds. Consider two arbitrary vertices $u \in M \setminus \{ v \}$ and $w \in Q \setminus N(v)$. Then we clearly have $vw \notin E(G)$. Since $M \setminus \{ v \} \subsetneq Q$ and $Q$ is a clique, we have $uw \in E(G)$. However, since $M$ is a clique module with $u,v \in M$ and $w \notin M$, this is a contradiction. So $\{ v \}$ is indeed an inclusion-wise maximal clique module. 
  Since $v$ has a non-neighbor in $Q$, its neighborhood $N(v)$ is a minimal separator. 
  Therefore, $\{ v \}$ is a moplex. Thus each vertex of $I$ is moplicial. 
 \end{proof}

 For any $v \in Q$, let $I_v = N(v) \cap I$. 
 
 \begin{claim} \label{IuIv}
  For any $u,v \in Q$, $u \ne v$, we have $|I_u \cup I_v| \le 3$.
 \end{claim}
 \begin{proof}
  \cref{I2} implies that $Q$ is a separator, and since $G$ is $t$-tough, we have
  \[ |I| = \omega(G-Q) \le \frac{|Q|}{t} . \]

  Let $u,v \in Q$ be two distinct vertices. Since $Q$ is a clique, we have $uv \in E(G)$. Let $S = S(uv) \subseteq V(G)$ be a vertex set guaranteed by \cref{claim:minttoughlemma}.
  Then, by \cref{claim:split_mintoughset}, we have
  \[ S = \big( Q \setminus \{ u, v \} \big) \cup \{ w \in I \mid uw,vw \in E(G) \} = \big( Q \setminus \{ u, v \} \big) \cup \big(I_u \cap I_v\big) . \]
  Then by \cref{claim:minttoughlemma}, we obtain
  \[ \frac{|Q| - 2 + |I_u \cap I_v|}{t} = \frac{|S|}{t} < \omega \big( (G-uv)-S \big) . \]
 
  Note that the components of $G-S$ are of two types: there is the component that contains $u$ and $v$ (and their neighbors they do not share), and there are single-vertex components formed by some vertices in $I$ (which are adjacent to neither $u$ nor $v$).
  Thus, by \cref{bridge}, we obtain
  \[ \omega \big( (G-uv)-S \big) = \omega(G-S) + 1 = \big( |I| - |I_u \cup I_v| + 1 \big) + 1 = |I| - |I_u \cup I_v| + 2 . \]
 
  By combining the above inequalities, we get
  \[ |I| \le \frac{|Q|}{t} = \frac{|Q| - 2 +|I_u \cap I_v|}{t} + \frac{2 - |I_u \cap I_v|}{t} < |I| - |I_u \cup I_v| + 2 + \frac{2 - |I_u \cap I_v|}{t} . \]
  Therefore, we have
  \[ |I_u \cup I_v| < 2 + \frac{2 - |I_u \cap I_v|}{t} \le 2 + \frac{2}{t} < 4 , \]
  where the last inequality is valid since $t > 1$. Hence, $|I_u \cup I_v| \le 3$.
 \end{proof}
 
 \begin{claim} \label{Iu2}
  There exists a vertex $u\in Q$ such that $|I_u| = 2$.
 \end{claim}
 \begin{proof}
  \cref{IuIv} implies that $|I_v| \le 3$ holds for every $v \in Q$. Moreover, if $|I_v| = 3$ for some $v \in Q$, then $|I| = 3$, implying that $v$ is a universal vertex, which contradicts \cref{thm:chordal_universal}. Therefore, $|I_v| \le 2$ holds for every $v \in Q$.

  Suppose to the contrary that $|I_v| \leq 1$ holds for all $v \in Q$. Then by \cref{I2} and by the connectivity of $G$, there exists $w \in Q$ with $I_w \ne \emptyset$. Let $x \in I_w$. Then $w$ is a maximum neighbor of $x$, contradicting \cref{Imoplex} and \cref{thm:no_mintough_chordal_graph_with_special_vertices}.
 \end{proof}
 
 Let $u \in Q$ be a vertex for which $|I_u| = 2$; by \cref{Iu2}, such a vertex exists.
 Let $I_u = \{ a,b \}$, and let $A = N(a)$ and $B = N(b)$. 
 Obviously, $u \in A \cap B$; thus $\{ a, b \} \subseteq N(A) \cap I$ and $\{ a, b \} \subseteq N(B) \cap I$. 
 In addition, $N(A) \cap I \ne \{ a, b \}$, otherwise $u$ would be a maximum neighbor of $a$, contradicting \cref{Imoplex} and \cref{thm:no_mintough_chordal_graph_with_special_vertices}.
 Similarly, $N(B) \cap I \ne \{ a, b \}$. 
 
 Let $c \in \big( N(A) \cap I \big)\setminus \{ a, b \}$.
 Then clearly, $c$ has a neighbor in $A$; let $v$ be such a neighbor of $c$.
 Now, $N(B) \cap I \ne \{ a, b, c \}$, otherwise $N(B) \cap I \ne \{ a, b \}$ and $I_u = \{ a, b \}$ would imply that $c$ has a neighbor $v'$ in $B \setminus \{ u \}$, and then, the edge $uv'$ would be a maximum neighboring edge of $b$, contradicting \cref{Imoplex} and \cref{thm:no_mintough_chordal_graph_with_special_vertices}.
 Let $d \in \big( N(B) \cap I \big)\setminus\{ a, b, c \}$.
 Then clearly, $d$ has a neighbor in $B$; let $w$ be such a neighbor of $d$.
 Now $a, c \in I_v$ and $b, d \in I_w$, where $a,b,c,d$ are four distinct vertices, contradicting~\cref{IuIv}.
\end{proof}

\end{document}